\documentclass[12pt,reqno] {amsart}%
\usepackage{amsfonts}
\usepackage{amsmath}
\usepackage{amssymb}
\usepackage{graphicx}%
\usepackage{hyperref}
\hypersetup{
    colorlinks = true,
    linkcolor = blue,
    urlcolor = blue,
    citecolor = blue,
    anchorcolor = black,
}

\usepackage{cleveref}
\usepackage{comment}
 \usepackage{mathrsfs}
 \usepackage{float}
 \usepackage{tikz}
 \usepackage{bm}
 \usepackage{ulem}
\usepackage{slashed} 
 \usepackage{braket} 
\usepackage{enumerate}
\usepackage{rotating}

\allowdisplaybreaks

\makeatletter
\def\l@section{\@tocline{1}{0pt}{1pc}{}{}}
\def\l@subsection{\@tocline{2}{0pt}{1pc}{4.6em}{}}
\def\l@subsubsection{\@tocline{3}{0pt}{1pc}{7.6em}{}}
\renewcommand{\tocsection}[3]{%
  \indentlabel{\@ifnotempty{#2}{\makebox[2.3em][l]{%
    \ignorespaces#1 #2.\hfill}}}#3}
\renewcommand{\tocsubsection}[3]{%
  \indentlabel{\@ifnotempty{#2}{\hspace*{2.3em}\makebox[2.3em][l]{%
    \ignorespaces#1 #2.\hfill}}}#3}
\renewcommand{\tocsubsubsection}[3]{%
  \indentlabel{\@ifnotempty{#2}{\hspace*{4.6em}\makebox[3em][l]{%
    \ignorespaces#1 #2.\hfill}}}#3}
\makeatother
\theoremstyle{plain}
\numberwithin{equation}{section}
\newtheorem{theorem}{Theorem}[section]
\newtheorem{corollary}[theorem]{Corollary}
\newtheorem{lemma}[theorem]{Lemma}
\newtheorem{proposition}[theorem]{Proposition}

\newtheorem{definition}[theorem]{Definition}
\newtheorem{remark}[theorem]{Remark}
\newtheorem{example}[theorem]{Example}

\newtheorem{main theorem}{Main Theorem}
\newtheorem{conjecture}[theorem]{Conjecture}

\newcommand{\be}{\begin{equation}}
\newcommand{\ee}{\end{equation}}

\renewcommand{\leq}{\leqslant}
\renewcommand{\geq}{\geqslant}

\begin{document}

\title{Equivalence of states and bimodule quantum channels}


\author{Linzhe Huang}
\address{Linzhe Huang, Beijing Institute of Mathematical Sciences and Applications, Beijing, 101408, China}
\email{huanglinzhe@bimsa.cn}


\author{Chunlan Jiang}
\address{Chunlan Jiang, School of Mathematical Sciences, Hebei Normal University, Shijiazhuang, Hebei, 050024, China}
\email{cljiang@hebtu.edu.cn}

\author{Zhengwei Liu}
\address{Zhengwei Liu, Yau Mathematical Sciences Center and Department of Mathematics, Tsinghua University, Beijing, 100084, China \\
 Beijing Institute of Mathematical Sciences and Applications, Beijing, 101408, China}
\email{liuzhengwei@mail.tsinghua.edu.cn}

\author{Jinsong Wu}
\address{Jinsong Wu, Beijing Institute of Mathematical Sciences and Applications, Beijing, 101408, China}
\email{wjs@bimsa.cn}

\maketitle

\begin{abstract}
We study the action of quantum channels on states on a von Neumann algebra $\mathcal{M}$ with conserved quantities encoded by a subalgebra $\mathcal{N}$. We define a phase as an equivalence class of normal states on $\mathcal{M}$ under the action of quantum channels preserving $\mathcal{N}$. First, we establish a relative quasi-entropic characterization of phase equivalence using the Petz recovery map. We further introduce the operator algebra (OA) symmetry of a state and define an equivalence relation between OA symmetries via bishifts of biprojections. When $\mathcal{N}\subseteq\mathcal{M}$ is an irreducible finite-index subfactor of type II$_1$, we prove that phase equivalence implies OA symmetry equivalence. Consequently, OA symmetry breaking provides a sufficient mechanism for detecting phase transitions. Furthermore, combining finite-index subfactor theory with quantum Fourier analysis, we prove that the cardinality of each phase is finite and obtain an explicit upper bound depending only on the Jones index.
\end{abstract}

{\bf Keywords.} Finiteness of phases, Jones index, bimodule quantum channel, operator algebra symmetry, Petz recovery map

{\bf MSC.} 46L37, 46L30, 46L60


\section{Introduction}
Phase transitions represent fundamental phenomena in mathematical physics, where macroscopic properties of a system change qualitatively as external parameters vary. The Landau theory of phase transitions provides a classical framework for understanding these phenomena through the concept of an order parameter and the symmetry structure of the underlying system \cite{Lan37}. In this framework, the free energy is constrained by the symmetry group of the Hamiltonian, and phase transitions are interpreted as spontaneous symmetry-breaking processes, characterized by a reduction of symmetry from a group to a subgroup \cite{GL50,Lan37}. Thus, symmetry breaking provides a fundamental mechanism for distinguishing different phases within the Landau paradigm.

Beyond the Landau paradigm, information-theoretic approaches have introduced entropy as a quantitative tool for characterizing quantum phases and phase transitions. Petz established a rigorous framework based on quantum relative entropy and recovery maps, in which the equality condition for the monotonicity of relative entropy is characterized by the existence of a recovery channel \cite{Pet88,Pet03}. The corresponding Petz recovery map provides a quantitative description of the reversibility of quantum processes and has subsequently found broad applications in quantum error correction, approximate quantum Markov chains, quantum channel reversibility, and the study of correlations in many-body quantum systems; see, e.g., \cite{BDL16,FR15,JRSWW18,SFR16}. This information-theoretic perspective provides a complementary way of characterizing quantum structures, in which entropy and recoverability replace local order parameters as fundamental quantities.

More recently, the study of topological phases has demonstrated that the Landau symmetry-breaking paradigm is not sufficient to classify all quantum phases of matter. Unlike conventional phases characterized by local order parameters, topological phases are distinguished by global topological invariants, long-range entanglement, and emergent non-local excitations, which may arise without conventional symmetry breaking \cite{GW09,Kit06,KP06,LW05,Wen95}. The mathematical description of such phases requires structures beyond ordinary group symmetry, where fusion categories and modular tensor categories provide natural frameworks for encoding anyonic excitations, fusion rules, and braiding statistics \cite{BK01,ENO05,Kit03,Tur94}. More generally, generalized, higher, and categorical symmetries have been developed to extend the conventional notion of symmetry beyond group actions and to provide broader frameworks for describing symmetry-breaking phases and topological orders \cite{GKSW15,KLWZZ20}. These developments suggest that more general notions of symmetry may be needed to capture quantum phases beyond the conventional Landau framework.

In this paper, we develop an operator-algebraic framework for studying phase transitions in quantum systems with conserved quantities. Our approach combines three complementary perspectives. First, we realize the entropy and recoverability framework of Petz within the setting of subfactor theory. Second, we introduce an operator algebra symmetry that extends conventional group symmetry and show that its breaking provides a sufficient mechanism for a phase transition. Third, and more importantly, we exploit the finite-index structure of subfactors to obtain quantitative bounds on the cardinality of phases.

More precisely, we consider an irreducible finite-index inclusion of factors

$$
    \mathcal{N}\subseteq\mathcal{M},
$$
where $\mathcal{M}$ is the von Neumann algebra generated by observables and $\mathcal{N}$ encodes the conserved quantities. The inclusion $\mathcal{N}\subseteq\mathcal{M}$ provides the basic algebraic structure through which the conserved quantities are incorporated into our framework.

Our first result provides a subfactor realization of the recoverability principle underlying Petz's theorem. Recall that Petz's sufficiency theorem characterizes the equality condition in the monotonicity of quantum relative entropy in terms of the existence of a recovery map \cite{Pet88,Pet03}. In our setting, the relevant quantum channels are required to be $\mathcal{N}$-bimodule channels, thereby incorporating the conserved quantities encoded by $\mathcal{N}$. 
We define phases in terms of equivalence under bimodule quantum channels. More precisely, two normal states $\varphi$ and $\psi$ on $\mathcal{M}$ are said to be equivalent if there exist bimodule quantum channels $\Phi$ and $\Psi$ such that
\begin{align*}
\varphi\Phi=\psi
\qquad\text{and}\qquad
\psi\Psi=\varphi.
\end{align*}
We then establish a relative quasi-entropic criterion characterizing this phase equivalence.
\begin{theorem}[Relative quasi-entropic criterion, see Theorem \ref{thm:relative quasi entropy}]
\label{thm:main theorem 1}
Suppose $\mathcal{N}\subseteq\mathcal{M}$ is an irreducible finite inclusion of factors and $\varphi,\psi$ are faithful normal states, with $0<\beta<1$. Suppose there exists a $\mathcal{N}$-bimodule quantum channel $\Phi$ such that $\varphi\Phi=\psi$. Then $\varphi$ is equivalent to $\psi$ if and only if
\begin{align*}
H_{\beta}(\varphi|\widetilde{\varphi})
=
H_{\beta}(\psi|\widetilde{\varphi}),
\end{align*}
where

$$
    \widetilde{\varphi}
    =
    \varphi\mathcal{E}_{\mathcal{N}}
$$
and $\mathcal{E}_{\mathcal{N}}$ is the conditional expectation from $\mathcal{M}$ onto $\mathcal{N}$.
\end{theorem}

Thus, Theorem~\ref{thm:main theorem 1} realizes the Petz recoverability principle in the presence of conserved quantities and provides an information-theoretic criterion for phase equivalence. 

Our second result introduces an {\sl operator algebra symmetry} (OA symmetry). In contrast to conventional symmetry, which is described by the action of a group on the observable algebra, our symmetry is encoded intrinsically by the operator-algebraic structure associated with the inclusion $\mathcal{N}\subseteq\mathcal{M}$. 
For a normal state $\varphi$, we associate an intermediate subfactor

$$
    \mathcal{N}\subseteq\mathcal{P}_{\varphi}\subseteq\mathcal{M},
$$
defined in terms of the conditional expectations that preserve $\varphi$. More precisely, we set
\begin{align*}
\mathcal{P}_{\varphi}
:=
\bigcap_{\substack{
\mathcal{N}\subseteq\mathcal{P}\subseteq\mathcal{M}\
\varphi\mathcal{E}_{\mathcal{P}}=\varphi
}}
\mathcal{P},
\end{align*}
where $\mathcal{E}_{\mathcal{P}}:\mathcal{M}\to\mathcal{P}$ denotes the conditional expectation onto the intermediate subfactor $\mathcal{P}$.
The resulting OA symmetry provides an operator-algebraic generalization of group symmetry; in particular, the conventional group-symmetry framework is recovered as a special case, see Remark~\ref{rem:group symmetry}.

We next relate phase equivalence to the operator algebra symmetry associated with a state. To compare the operator algebra symmetries associated with different states, we introduce the following equivalence relation. We say that $P_\varphi$ and $P_\psi$ are equivalent if there exists a bishift of biprojection $V$ such that
$$
    V^*V=P_\varphi
    \qquad\text{and}\qquad
    VV^*=P_\psi.
$$

The connection between phase equivalence and OA symmetry is then captured by the following result.
\begin{theorem}[OA symmetry and phase equivalence, see Theorem \ref{thm:state equivalent}]
\label{thm:main theorem 2}
Suppose $\mathcal{N}\subseteq\mathcal{M}$ is an irreducible finite inclusion of factors, and let $\varphi,\psi$ be normal states. If $\varphi$ is equivalent to $\psi$, then
    \begin{align*} \mathcal{P}_{\varphi} \text{ is equivalent to } \mathcal{P}_{\psi}. \end{align*}
\end{theorem}

Theorem~\ref{thm:main theorem 2} shows that OA symmetry is an invariant of the phase. Equivalently, its contrapositive gives a sufficient mechanism for detecting phase transitions:

$$
\boxed{
\text{OA symmetry breaking}
    \quad\Longrightarrow\quad
    \text{phase transition}.
}
$$

In this sense, OA symmetry provides an operator-algebraic extension of the symmetry-breaking mechanism in the Landau paradigm. We emphasize, however, that Theorem~\ref{thm:main theorem 2} does not imply the converse: a phase transition need not be accompanied by OA symmetry breaking. Thus, OA symmetry should be viewed as a sufficient, rather than necessary, mechanism for detecting phase transitions.

While Theorems~\ref{thm:main theorem 1} and \ref{thm:main theorem 2} provide information-theoretic and symmetry-based descriptions of phase structure, our main quantitative result concerns the size of a phase. In the finite-index setting, the operator-algebraic structure imposes strong finiteness constraints. This leads naturally to the following question:

$$
\textit{How large can a phase be?}
$$

The finite-index theory of subfactors, together with quantum Fourier analysis, provides the two structural ingredients needed to address this question. On the one hand, the lattice $\mathcal{L}(\mathcal{N}\subseteq\mathcal{M})$ of intermediate subfactors is finite \cite{Wat96}, and the quantitative estimate in \cite[Theorem~1.2]{BDLR19} gives a bound depending only on the Jones index. More recently, Azzouz, Ghosh, and Palcoux \cite[Corollary 1.5]{AGP26} obtained the improved estimate

$$
 \#\mathcal{L}(\mathcal{N}\subseteq\mathcal{M})
 \leq C
 :=
 2^{\dim_{\mathbb{C}}(\mathcal{N}'\cap\mathcal{M}_1)-1}
 \leq
 2^{\lfloor [\mathcal{M}:\mathcal{N}] \rfloor-1}.
$$

On the other hand, quantum Fourier analysis provides the additional structure needed to relate phases to intermediate subfactors. In particular, the properties of bishifts of biprojections allow us to control the cardinality of a phase in terms of the lattice of intermediate subfactors. Combining these two ingredients, we obtain the following main quantitative result.

\begin{theorem}[Finiteness of phases, see Theorem \ref{thm:phase finite upper bound}]
\label{thm:main theorem 3}
Suppose $\mathcal{N}\subseteq\mathcal{M}$ is an irreducible finite inclusion of factors and $\varphi$ is a normal state. Let $[\varphi]$ be the phase containing $\varphi$. Then
\begin{align*}
\#[\varphi]
\leq
[\mathcal{M}:\mathcal{N}]C.
\end{align*}
\end{theorem}

Theorem~\ref{thm:main theorem 3} is the principal quantitative consequence of our framework. It shows that the finite-index condition imposes a universal constraint on the cardinality of a phase. In particular, although the space of normal states is generally infinite-dimensional, the cardinality of each phase is controlled by the finite-index operator-algebraic structure. The resulting bound depends explicitly on the Jones index and on the size of the lattice of intermediate subfactors.

The bound in Theorem~\ref{thm:main theorem 3} can be compared with the conventional group-symmetry setting. Let $G$ be a finite group acting outerly on $\mathcal{N}$ and consider the associated group subfactor

$$
    \mathcal{N}\subseteq\mathcal{N}\rtimes G.
$$

In this case,

$$
    [\mathcal{N}\rtimes G:\mathcal{N}]=|G|,
$$

and the corresponding phase cardinality satisfies the sharper bound

$$
    \#[\varphi]\leq |G|.
$$

This example suggests that the general bound in Theorem~\ref{thm:main theorem 3} may admit a substantial strengthening. It motivates the following conjecture.

\begin{conjecture}
Suppose $\mathcal{N}\subseteq\mathcal{M}$ is an irreducible finite inclusion of factors and $\varphi$ is a normal state. Let $[\varphi]$ be the phase containing $\varphi$. Then
\begin{align*}
\#[\varphi]\leq[\mathcal{M}:\mathcal{N}].
\end{align*}
\end{conjecture}

The group-subfactor example illustrates that our operator-algebraic framework contains the conventional group-symmetry setting while allowing for a broader class of symmetry structures arising intrinsically from subfactor inclusions. Taken together, our results provide three complementary perspectives on quantum phase structure:

$$
\text{entropy and recoverability},
\qquad
\text{operator-algebraic symmetry},
\qquad
\text{quantitative finiteness of phases}.
$$

The first provides an information-theoretic criterion for phase equivalence, the second identifies OA symmetry breaking as a sufficient mechanism for phase transitions, and the third establishes a quantitative finiteness principle for phases. Among these, Theorem~\ref{thm:main theorem 3} provides the principal quantitative contribution of the paper.

The paper is organized as follows. In \S\ref{sec:preliminary}, we present the necessary preliminaries from operator algebras, subfactor theory, and quantum Fourier analysis. In \S\ref{sec:phase}, we define phases in terms of $\mathcal{N}$-bimodule quantum channels and establish a relative quasi-entropic criterion for phase equivalence using the Petz recovery map, thereby proving Theorem~\ref{thm:main theorem 1}. In \S\ref{sec:symmetries}, we introduce OA symmetries and define an equivalence relation among OA symmetries via bishifts of biprojections. Finally, in \S\ref{sec:finiteness of phases}, we show that phase equivalence implies OA symmetry equivalence, yielding the symmetry-breaking criterion for phase transitions in Theorem~\ref{thm:main theorem 2}. We then investigate the finiteness of phases and establish the quantitative bound in Theorem~\ref{thm:main theorem 3}.

\section{Preliminary Results}\label{sec:preliminary}
Let $\mathcal{M}$ be a von Neumann algebra and let $\mathcal{N}$ be a subalgebra of $\mathcal{M}$.
If $\mathcal{M}$ has a trivial center, then $\mathcal{M}$ is called a factor.
We call an inclusion of von Neumann algebras $\mathcal{N}\subseteq\mathcal{M}$ a {\sl finite inclusion of factors} if the index $[\mathcal{M}:\mathcal{N}]$ is finite and $\mathcal{M}$ and $\mathcal{N}$ are finite factors.
The inclusion is called irreducible if $\mathcal{N}'\cap\mathcal{M}=\mathbb{C}\mathbf{1}$.

\subsection{Fourier transform}
Let $\mathcal{N}\subseteq\mathcal{M}$ be a finite inclusion of factors and $\mathcal{E}_{\mathcal{N}}$ the  conditional expectation from $\mathcal{M}$ onto $\mathcal{N}$.
The conditional expectation $\mathcal{E}_{\mathcal{N}}$ induces a projection $e_{\mathcal{N}}$.
The von Neumann algebra generated by $\mathcal{M}$ and $e_{\mathcal{N}}$ is a factor again, denoted by $\mathcal{M}_1$.
Iterating the process, we obtain the Jones tower as follows:
\begin{align*}
    \mathcal{N}\subset\mathcal{M}\subset\mathcal{M}_1\subset\mathcal{M}_2\subset\cdots
\end{align*}
The relative commutants $\mathcal{N}'\cap\mathcal{M}_k$ and $\mathcal{M}'\cap\mathcal{M}_k$ turn out to be critical in characterizing the structure of the subfactors. 
The following
\begin{align*}
\mathcal{N}'\cap\mathcal{N}\quad\subset\quad &\mathcal{N}'\cap\mathcal{M}\quad\subset\quad {\mathcal{N}'\cap\mathcal{M}_1}\quad\subset\quad \mathcal{N}'\cap\mathcal{M}_2\quad\subset\quad\cdots\\
&\quad\ \cup\quad\quad\quad\quad\quad\quad\cup
\quad\quad\quad\quad\quad \quad\quad \cup
\\
 &\mathcal{M}'\cap\mathcal{M}\quad\subset\quad \mathcal{M}'\cap\mathcal{M}_1\quad\subset\quad {\mathcal{M}'\cap\mathcal{M}_2}\quad\subset\quad\cdots
\end{align*}
is called the standard invariant of $\mathcal{N}\subset\mathcal{M}$.
There are three axiomatizations to study standard invariant: Jones planar algebras, Popa’s $\lambda$-lattice and Ocneanu’s paragroup.
Jones planar algebra is a powerful and handy computation tool. 
We denote by $\mathscr{P}^{\mathcal{N}\subset\mathcal{M}}$ the associated subfactor planar algebra and $tr_k$ the unnormalized trace on $\mathcal{N}'\cap\mathcal{M}_{k-1}$ or $\mathcal{M}'\cap\mathcal{M}_k$.

There is a Fourier transform $\mathfrak{F}$ between the pair $\{\mathcal{N}'\cap\mathcal{M}_{k},\mathcal{M}'\cap\mathcal{M}_{k+1} \}$ for each $k\geq1$.
In particular, when $k=1$, the Fourier transform and its inverse are given by
  \begin{align*}
    \mathfrak{F}(x)=&\mu^{3/2} E_{\mathcal{M}'\cap\mathcal{M}_2}(\Delta xe_2e_1), \quad x\in \mathcal{N}'\cap \mathcal{M}_1\\
     \mathfrak{F}^{-1}(x)=& \mu^{3/2} E_{\mathcal{M}_1}(\Delta^{-1} x e_1e_2) \quad x\in \mathcal{M}'\cap \mathcal{M}_2,
     \end{align*}
     see e.g. \cite{BJ00}.
In planar algebras \cite{Jon21}, the Fourier transform $\mathfrak{F}$: $\mathscr{P}_{2,+}\to\mathscr{P}_{2,-}$ is a $90^\circ$ rotation:
\begin{align}\label{eq:Fourier transform}
\mathfrak{F}(x):=\raisebox{-0.9cm}{
\begin{tikzpicture}[scale=1.5]
\path [fill=gray!40] (-0.3, -0.4) rectangle (0.8, 0.9);
\path [fill=white] (0.35, -0.4)--(0.35, 0.5) .. controls +(0, 0.3) and +(0, 0.3) .. (0.65, 0.5)--(0.65, -0.4);
\path[fill=white] (0.15, 0.9) -- (0.15, 0) .. controls +(0, -0.3) and +(0, -0.3) .. (-0.15, 0)--(-0.15, 0.9);
\draw [blue, fill=white] (0,0) rectangle (0.5, 0.5);
\node at (0.25, 0.25) {$x$};
\draw (0.35, 0)--(0.35, -0.4) (0.15, 0.5)--(0.15, 0.9);
\draw (0.35, 0.5) .. controls +(0, 0.3) and +(0, 0.3) .. (0.65, 0.5)--(0.65, -0.4);
\draw (0.15, 0) .. controls +(0, -0.3) and +(0, -0.3) .. (-0.15, 0)--(-0.15, 0.9);
\end{tikzpicture}}
\;.
\end{align}
     
\subsection{Some properties of bimodule quantum channels}
To evolve states of matters, we need unital normal completely positive maps $\Phi:\mathcal{M}\to \mathcal{M}$ which preserve states.
To ensure that a unital normal completely positive map $\Phi$ preserves the quantities from $\mathcal{N}$, we require that $\Phi$ is a bimodule completely positive map, i.e.
\begin{align*}
    \Phi(y_1 x y_2)=y_1\Phi(x) y_2, \quad\text{for all } x\in\mathcal{M},\ y_i\in\mathcal{N},\ i=1,2.
\end{align*}
We call a unital normal completely positive map as a {\sl quantum channel} and a unital normal completely positive
bimodule map as a {\sl bimodule quantum channel}.

There is a one-to-one correspondence between bounded bimodule maps and elements in $\mathcal{N}'\cap\mathcal{M}_1$.
To describe complete positiveness, we introduce the notion of $\mathfrak{F}$-positive.
Recall that  an element $T\in\mathcal{N}'\cap\mathcal{M}_1$ is $\mathfrak{F}$-positive if $\mathfrak{F}(T)$ is positive (see \cite[Definition 3.2]{HJLW25b}). 
A bimodule quantum channel $\Phi$ can be viewed as an $\mathfrak{F}$-positive element $T$ in $\mathcal{N}'\cap\mathcal{M}_1$.
We will denote $\mathfrak{F}(T)$ by $\mathfrak{F}(\Phi)$ and $T$ by $\Phi$ if there is no confusion. 
For each $\mathfrak{F}$-positive element $V$ in $\mathcal{N}'\cap\mathcal{M}_1$ with $\|\mathfrak{F}(V)\|_1=[\mathcal{M}:\mathcal{N}]^{1/2}$, we denoted by $\Phi_V$ the bimodule quantum channel associated to $V$, where $\|\cdot\|_1$ is the
$1$-norm induced by $tr_2$.
More precisely,
\begin{align*}
    \Phi_V(x)=[\mathcal{M}:\mathcal{N}]\mathcal{E}_{\mathcal{M}}(Vxe_{\mathcal{N}}), \quad x\in\mathcal{M},
\end{align*}
where $\mathcal{E}_{\mathcal{M}}$ is the conditional expectation from $\mathcal{M}_1$ onto $\mathcal{M}$, see \cite[Proposition 2.3]{HJLW25}.
Thanks to the pictorial characterization of subfactors, there are horizontal multiplication (the convolution) and the vertical multiplication (the usual multiplication) on  $\mathcal{N}'\cap\mathcal{M}_1, \mathcal{M}'\cap\mathcal{M}_2$.
The composition of bimodule quantum channels is equivalent to the convolution of the corresponding positive elements in $\mathcal{M}'\cap\mathcal{M}_2$.

We recall some results for bimodule quantum channels.
\begin{proposition}[Theorem 4.4 in \cite{HJLW25}]\label{prop:conditional expectation}
    Suppose that $\mathcal{N}\subseteq\mathcal{M}$ is an irreducible finite inclusion of factors and $\Phi$ is a bimodule quantum channel. 
    Then there exists an intermediate subfactor $\mathcal{P}$ such that
\begin{align*}
    \mathcal{E}_{\mathcal{P}}=\lim_{n\to\infty}\frac{1}{n}\sum_{k=1}^n\Phi^k,
\end{align*}
where $ \mathcal{E}_{\mathcal{P}}$ is the conditional expectation from $\mathcal{M}$ onto $\mathcal{P}$.
\end{proposition}
Suppose $\Phi$ is a bimodule quantum channel. 
The maps $\Phi\Phi^*$ and $\Phi^*\Phi$ are bimodule quantum channels again, where $\Phi^*$ is to take the adjoint of the corresponding element in $\mathcal{N}'\cap\mathcal{M}_1$. 
By Proposition \ref{prop:conditional expectation},  there exist intermediate subfactors $\mathcal{P}$ and $\mathcal{Q}$ such that
\begin{align}\label{eq:support}
    \mathcal{E}_{\mathcal{P}}=\lim_{n\to\infty}\frac{1}{n}\sum_{k=1}^n(\Phi\Phi^*)^k,\quad \mathcal{E}_{\mathcal{Q}}= \lim_{n\to\infty}\frac{1}{n}\sum_{k=1}^n(\Phi^*\Phi)^k.
\end{align}
It is clear that these two conditional expectations have the following relations:
\begin{align}\label{eq:relation between two conditional expectation}
    \mathcal{E}_{\mathcal{P}}=\Phi \mathcal{E}_{\mathcal{Q}}\Phi^*,\quad \mathcal{E}_{\mathcal{Q}}=\Phi^*\mathcal{E}_{\mathcal{P}}\Phi.
\end{align}
\subsection{Bishifts of biprojections}
A projection $P$ in $\mathcal{N}'\cap\mathcal{M}_1$ is a biprojection \cite{Bis94,BJ97} if $\mathfrak{F}(P)$ is a multiple of a projection.
There is a one-to-one correspondence between intermediate subfactors $\mathcal{P}$ with minimal conditional expectations $\mathcal{E}_{\mathcal{P}}$ and biprojections $P$. 
The equivalence of projections is described by partial isometries. 
However, the equivalence of biprojections is described by extremal bipartial isometries, i.e. bishifts of biprojections.
This helps us to define the equivalence of intermediate subfactors (i.e. the symmetries). 

Due to the significance of bishifts of biprojections in characterizing the phase, we recall their definitions (see also \cite[Definition 6.6]{JLW16}).
Suppose $P$ is a biprojection in $\mathcal{N}'\cap\mathcal{M}_1$, we denote by $\widetilde{P}$ the range projection
of $\mathfrak{F}(P)\in\mathcal{M}'\cap\mathcal{M}_2$.
A nonzero element $x$ in $\mathcal{N}'\cap\mathcal{M}_1$ is said to be a bishift of a biprojection $P$ if there
exist a right shift $P_g$ of the biprojection $P$ and a right shift $\widetilde{P}_h$ of the biprojection $\widetilde{P}$
and an element $y\in \mathcal{N}'\cap\mathcal{M}_1$ such that $x=\mathfrak{F}(\widetilde{P}_h)\ast (yP_g)$.
There are 8 different ways to construct bishifts of biprojections.

Bishifts of biprojections can be characterized as the extremizers of quantum uncertainty principles.
Recall that an element $x$ in $\mathcal{N}'\cap\mathcal{M}_1$ or $\mathcal{M}'\cap\mathcal{M}_2$ is called extremal if $x$ is an extremizer of the Hausdorff-Young inequality:
\begin{align*}
    \|\mathfrak{F}(x)\|_{\infty}=\frac{\|x\|_1}{[\mathcal{M}:\mathcal{N}]^{1/2}}.
\end{align*}
See \cite[Definition 6.1]{JLW16}.
\begin{theorem}(See \cite[Main Theorem 2]{JLW16})\label{thm:bishift}
Suppose $\mathcal{N}\subseteq\mathcal{M}$ is an irreducible finite inclusion of factors.
For a non-zero element $x\in\mathcal{N}'\cap\mathcal{M}_1$, the following statements are equivalent:
\begin{enumerate}
    \item $\mathcal{S}(x)\mathcal{S}(\mathfrak{F}(x))=[\mathcal{M}:\mathcal{N}]$, where $\mathcal{S}(x)$ is the support of the range projection of $x$.
    \item $H(|x|^2)+H(|\mathfrak{F}(x)|^2)=\|x\|_2^2(\log[\mathcal{M}:\mathcal{N}]-4\log\|x\|_2)$, where $H(|x|^2)$ is the von Neumann entropy of $|x|^2$.
    \item $x$ is an extremal bipartial isometry.
    \item  $x$ is a partial isometry and $\mathfrak{F}^{-1}(x)$ is extremal.
\item  $x$ is a bishift of a biprojection.
\end{enumerate}
\end{theorem}
Let $G$ be a finite group and $\mathcal{N}\subseteq\mathcal{N}\rtimes G$ the group subfactor.
An element $x$ in the group algebra $\mathcal{L}(G)$ is a bishift of biprojection if and only if $x=c\sum_{h\in H}\chi(h)hg$, where $H$ is a subgroup of $G$ and $\chi$ is a one-dimensional representation of $H$, $g\in G$, $c\in\mathbb{C}$, see \cite[Proposition 8.1]{JLW16}.
\subsection{Tomita-Takesaki theory}
We recall some definitions in Tomita-Takesaki theory that will be frequently used in the following sections.
We assume that $\varphi$ is a normal faithful state on $\mathcal{M}$, i.e., an equilibrium state.
The closure of $*$-operation on $\mathcal{M}$ is denoted by $S_{\varphi}$.
The modular operator is denoted by $\Delta_{\varphi}=S_{\varphi}^*S_{\varphi}$.
The modular automorphism group is denoted by $\sigma^{\varphi}$.
Let $\mathcal{E}_{\mathcal{N}}$ be the minimal condition expectation from $\mathcal{M}$ onto $\mathcal{N}$.
Denote by $\widetilde{\varphi}=\varphi\mathcal{E}_{\mathcal{N}}$.
The relative modular operator of  $\widetilde {\varphi}$ and $\varphi$ is denoted by $\Delta_{\widetilde{\varphi},\varphi}$.
The modular Hamiltonian $H_{\varphi}$ is $-\log \Delta_{\widetilde{\varphi},\varphi}$.
The physical Hamiltonian $K_{\varphi}$ at inverse temperature $\beta$ is 
\begin{align*}
    K_{\varphi}=\beta^{-1} (H_{\varphi}-\log[\mathcal{M}:\mathcal{N}])=\beta^{-1}(-\log \Delta_{\widetilde{\varphi},\varphi}-\log[\mathcal{M}:\mathcal{N}]).
\end{align*}
We refer to \cite{Lon18,Lon20} for more details on modular Hamiltonian, physical Hamiltonian etc.

\section{Phase}\label{sec:phase}
The phases of matter encountered in everyday life include gases, liquids, solids, and plasma.
Phase transitions involve changes between these states, such as freezing, melting, sublimation, vaporization, condensation,
deposition, ionization, and recombination.
This topic is fundamental in condensed matter physics and statistical mechanics.
The concept of phase transitions was first proposed by Ehrenfest \cite{Ehr33} following the discovery of a surprising phase transition in liquid helium. 
More recent definitions of phase transitions can be found in works such as those by Stanley (1987) \cite{Sta72}, and Plischke and Bergersen (2006) \cite{PLM02}. 
Various parameters, including temperature, pressure, magnetic field, and electric field, can drive a phase transition.
The phase can be viewed as an equivalent class of states of matter and the parameters such as temperature, pressure can be viewed as certain channels interchanging states of matters.

In this section, we will first review the Petz recovery map and the relative quasi-entropy.
Next we consider the reformulation of the Petz recovery map in bimodule case and  give an equivalent description of the equivalent class of a state by relative quasi-entropies.
Finally, we discuss the free energy in subfactors.

\subsection{Petz recovery map}
Suppose $\mathcal{M}$ is a von Neumann algebra and $\Phi$ is a faithful quantum channel on $\mathcal{M}$.
Let $\varphi$ be a normal faithful state on $\mathcal{M}$.
Then $\varphi\Phi$ is also a normal faithful state on $\mathcal{M}$.
Let $\Omega_{\varphi},\Omega_{\varphi\Phi}$ be the unit vectors obtained by the GNS representations of $\mathcal{M}$ induced by $\varphi,\varphi\Phi$ respectively.
The Petz recovery map \cite{Pet88} with respect to the state $\varphi$ and the quantum channel $\Phi$ is the following.
\begin{definition}[Petz Recovery Map]
Suppose $\mathcal{M}$ is a von Neumann algebras and $\Phi$ is a faithful quantum channel on $\mathcal{M}$.
Let $\varphi$ be a normal faithful state on $\mathcal{M}$.
Then the Petz recovery map $\Phi^*_{\varphi}$ on $\mathcal{M}$ is defined as follows
\begin{align}
    \langle \Phi(x) \Omega_{\varphi},Jy\Omega_{\varphi}\rangle =\langle x\Omega_{\varphi\Phi},J\Phi^*_{\varphi}(y)\Omega_{\varphi\Phi}\rangle,\quad x,y\in\mathcal{M}.
\end{align}
\end{definition}
\begin{remark}
    We have that the Petz recovery map $\Phi_{\varphi}^*$ is a quantum channel if $\Phi$ is a quantum channel.
    Moreover, 
    \begin{align*}
        \varphi \Phi\Phi_{\varphi}^*=\varphi.
    \end{align*}
\end{remark}
Now let $\varphi, \rho$ be two faithful normal states.
 They are called equivalent if there exist quantum channels $\Phi$ and $\Psi$ such that $\varphi\Phi=\rho$ and $\rho=\varphi\Psi$.
By Petz's result, we have that $\varphi$ and $ \rho$ are equivalent if and only if $\varphi\Phi=\rho$.
We will see that  there is a slight difference in bimodule case in \S \ref{subsec:bimodule petz recovery}.

Next we consider the relative quasi-entropy for two normal states.
Let $\varphi,\rho$ be two normal states  with  $supp(\varphi)\leq supp(\rho)$ and let $\Delta_{\varphi,\rho}$ be the relative modular operator.
The relative quasi-entropy $H_{\beta}(\varphi\|\rho)$, $0<\beta<1$, is defined as follows 
\begin{align*}
    H_{\beta}(\varphi\|\rho)=\langle\Delta_{\varphi,\rho}^{\beta} \Omega_{\rho},\Omega_{\rho}\rangle.
\end{align*}
The following lemma can be proved by \cite{Pet88} and \cite{Pet03}.
For completeness, we give the proof here.
\begin{lemma}\label{lem:beta}
    Suppose $\mathcal{M}$ is a von Neumann algebra and $\Phi$ is a faithful quantum channel on $\mathcal{M}$.
    Suppose $\varphi,\rho$ are faithful normal states, $0<\beta<1$.
    Then $H_{\beta}(\varphi\|\rho)=H_{\beta}(\varphi\Phi\|\rho\Phi)$ if and only if $H_{1/2}(\varphi\|\rho)=H_{1/2}(\varphi\Phi\|\rho\Phi)$.
\end{lemma}
\begin{proof}
For simplicity, we write $\Delta=\Delta_{\varphi,\rho}$ and $\Delta_0=\Delta_{\varphi\Phi,\rho\Phi}$.
    First, we use the formula
\begin{align*}
    x^{\beta}=\frac{\sin \pi \beta}{\pi}\int_{0}^{\infty}t^{\beta-1}-t^{\beta}(x+t)^{-1}dt
\end{align*}
and obtain that
\begin{align*}
    H_{\beta}(\varphi\|\rho)=\frac{\sin \pi \beta}{\pi}\int_{0}^{\infty}t^{\beta-1}-t^{\beta}\langle (\Delta+t)^{-1}\Omega_{\rho},\Omega_{\rho} \rangle dt.
\end{align*}
Similarly,
\begin{align*}
    H_{\beta}(\varphi\Phi\|\rho\Phi)=\frac{\sin \pi \beta}{\pi}\int_{0}^{\infty}t^{\beta-1}-t^{\beta}\langle (\Delta_{0}+t)^{-1}\Omega_{\rho\Phi},\Omega_{\rho\Phi} \rangle dt.
\end{align*}
By the proof of \cite[Theorem 3]{Pet88}, we have
\begin{align*}
    \langle (\Delta_{0}+t)^{-1}\Omega_{\rho\Phi},\Omega_{\rho\Phi} \rangle\leq \langle (\Delta+t)^{-1}\Omega_{\rho},\Omega_{\rho} \rangle.
\end{align*}
Thus $H_{\beta}(\varphi\|\rho)=H_{\beta}(\varphi\Phi\|\rho\Phi)$ implies that $  \langle (\Delta_{0}+t)^{-1}\Omega_{\rho\Phi},\Omega_{\rho\Phi} \rangle= \langle (\Delta+t)^{-1}\Omega_{\rho},\Omega_{\rho} \rangle$ for almost $t\in\mathbb{R}^+$ and  by continuity for all $t\in\mathbb{R}^+$.
Hence $H_{1/2}(\varphi\|\rho)=H_{1/2}(\varphi\Phi\|\rho\Phi)$.
The converse direction is similar.
\end{proof}
\begin{theorem}[Theorem 3 in \cite{Pet88}]\label{thm:petz}
     Suppose $\mathcal{M}$ is a von Neumann algebra and $\Phi$ is a faithful quantum channel on $\mathcal{M}$.
    Suppose $\varphi,\rho$ are faithful normal states.
    Then $\Phi_{\varphi}^*=\Phi_{\rho}^*$ if and only if $H_{1/2}(\varphi\|\rho)=H_{1/2}(\varphi\Phi\|\rho\Phi)$, and hence if and only  if $H_{\beta}(\varphi\|\rho)=H_{\beta}(\varphi\Phi\|\rho\Phi)$, $0<\beta<1$, by Lemma \ref{lem:beta}.
\end{theorem}

\subsection{Bimodule case}\label{subsec:bimodule petz recovery}
In many physic systems, there are lots of conversed quantity, which is preserved by the channels. 
Due to this reason, we consider an inclusion  $\mathcal{N}\subseteq\mathcal{M}$ of von Neumann algebras and an $\mathcal{N}$-bimodule quantum channel $\Phi$ on $\mathcal{M}$.
When the inclusion is an irreducible finite inclusion of factors, we have that $\Phi$ is faithful and $\mathcal{E}_{\mathcal{N}}\Phi=\Phi\mathcal{E}_{\mathcal{N}}=\mathcal{E}_{\mathcal{N}}$.
Therefore, for any normal faithful state $\varphi$ on $\mathcal{M}$, let $\widetilde{\varphi}=\varphi\mathcal{E}_{\mathcal{N}}$, we have that $\widetilde{\varphi}=\widetilde{\varphi}\Phi$ for any bimodule quantum channel $\Phi$.
Moreover, the Petz recovery map $\Phi_{\varphi}^*$ has a graphical representation in planar algebras.
Therefore, we assume that $\mathcal{N}\subseteq\mathcal{M}$ is an irreducible finite inclusion of factors in the following contents.

The equivalent class of states is described by bimodule quantum channels.
We introduce the following notion of phase.
\begin{definition}
    Suppose that $\mathcal{N}\subseteq\mathcal{M}$ is an inclusion of von Neumann algebras and $\varphi,\psi$ are normal states on $\mathcal{M}$.
    We say $\varphi,\psi$ are equivalent if there exist bimodule quantum channels  $\Phi,\Psi$ on $\mathcal{M}$ such that
    \begin{align*}
        \varphi\Phi=\psi,\quad \psi \Psi=\varphi.
    \end{align*}
    We denote by $[\varphi]$ the equivalent class and call it as a phase for quantum system $\mathcal{N}\subseteq\mathcal{M}$.
\end{definition}
\begin{remark}
    Suppose $\varphi,\psi$ are normal states on $\mathcal{M}$ and $\varphi$ is equivalent to $\psi$.
    We have that $\varphi|_{\mathcal{N}}=\psi|_{\mathcal{N}}$, i.e. the information of $\varphi$ and $\psi$ coincide on the subsystem $\mathcal{N}$.
\end{remark}

\begin{remark}\label{rem:channel faithful}
    Suppose $\mathcal{N}\subseteq\mathcal{M}$ is an irreducible finite inclusion of factors and $\varphi$ is a normal state.
Let $\psi=\varphi\Phi$ and $\Phi^*_{\varphi}$ the Petz recovery map of $\Phi$ relative to $\varphi$.  
Then we have that $\varphi \Phi \Phi_{\varphi}^*=\varphi$.
Hence $\psi \Phi_{\varphi}^*=\varphi$.
However $\Phi_{\varphi}^*$ is not a bimodule quantum channel in general.
\end{remark}

\begin{remark}\label{rem:bimodule condition}
Suppose $\Phi:\mathcal{M}\to \mathcal{M}$ is a faithful  quantum channel and $\varphi$ is a faithful normal state on $\mathcal{M}$. 
By \cite[Theorem 3.18 and Corollary 3.19]{Pau02}, $\Phi^*_{\varphi}$ is a bimodule quantum channel if and only if $\Phi^*_{\varphi}(y)=y$ for any $y\in\mathcal{N}$.
Then by \cite[Theorem 2]{Pet88}, we have that $\Phi^*_{\varphi}$ is a bimodule quantum channel if and only if $\Phi(\sigma_t^{\varphi\Phi}(y) )=\sigma_t^{\varphi}(y)$ for any $y\in\mathcal{N}$ where $\sigma_t^{\varphi}$ is the modular automorphism group of $\varphi$.
\end{remark}

Though the Petz recovery map $\Phi^*_{\varphi}$ may not be a bimodule quantum channel, the Petz recovery $\Phi_{\widetilde{\varphi}}^*$ is a bimodule quantum channel.
\begin{lemma}\label{lem:adjont bimodule map}
    Suppose $\mathcal{N}\subseteq\mathcal{M}$ is an irreducible finite inclusion of factors and $\varphi$ is a faithful normal state on $\mathcal{M}$, $\Phi$ is a bimodule quantum channel.  
    Then $\Phi_{\widetilde{\varphi}}^*$ is a bimodule quantum channel.
\end{lemma}
\begin{proof}
For any $x,y\in\mathcal{M}$,
\begin{align*}
   \widetilde{\varphi}(\mathcal{E}_{\mathcal{N}}(x)y^*) =\langle \mathcal{E}_{\mathcal{N}}(x)\Omega_{\widetilde{\varphi}},y\Omega_{\widetilde{\varphi}}\rangle=\langle x\Omega_{\widetilde{\varphi}},\mathcal{E}_{\mathcal{N}}^*( y)\Omega_{\widetilde{\varphi}}\rangle=\widetilde{\varphi}(x\mathcal{E}_{\mathcal{N}}^*( y)^*).
\end{align*}
Since $\widetilde{\varphi}\mathcal{E}_{\mathcal{N}}=\widetilde{\varphi}$, we have 
\begin{align*}
    \widetilde{\varphi}(x\mathcal{E}_{\mathcal{N}}(y^*))=\widetilde{\varphi}(\mathcal{E}_{\mathcal{N}}(x)y^*)=\widetilde{\varphi}(x\mathcal{E}_{\mathcal{N}}^*( y)^*).
\end{align*}
Thus $\mathcal{E}_{\mathcal{N}}=\mathcal{E}_{\mathcal{N}}^*$.
    For any $x\in\mathcal{M}$, we have that
    \begin{align*}
S_{\widetilde{\varphi}}\mathcal{E}_{\mathcal{N}}x\Omega_{\widetilde{\varphi}}=\mathcal{E}_{\mathcal{N}}S_{\widetilde{\varphi}}x\Omega_{\widetilde{\varphi}},
    \end{align*}
    i.e. $\mathcal{E}_{\mathcal{N}}S_{\widetilde{\varphi}}\subset S_{\widetilde{\varphi}}\mathcal{E}_{\mathcal{N}}$.
This implies that  $\mathcal{E}_{\mathcal{N}}S_{\widetilde{\varphi}}^*\subset S_{\widetilde{\varphi}}^*\mathcal{E}_{\mathcal{N}}$.
Moreover, $\mathcal{E}_{\mathcal{N}}\Delta_{\widetilde{\varphi}}^{it}\subset\Delta_{\widetilde{\varphi}}^{it}\mathcal{E}_{\mathcal{N}}$. 
    Hence $\sigma_t^{\widetilde{\varphi}}(y)\in\mathcal{N}$ for any $y$ in $\mathcal{N}$.
Since $\Phi$ is a bimodule quantum channel, $\Phi(\sigma_t^{\widetilde{\varphi}}(y))=\sigma_t^{\widetilde{\varphi}}(y)$ for any $y$ in $\mathcal{N}$.
By Remark \ref{rem:bimodule condition}, $\Phi_{\widetilde{\varphi}}^*$ is a bimodule quantum channel.
\end{proof}
By Petz's results, we have that $\Phi_{\varphi}^*=\Phi_{\widetilde{\varphi}}^*$ if and only if $H_{\beta}(\varphi\|\widetilde{\varphi})= H_{\beta}(\varphi\Phi\|\widetilde{\varphi}\Phi)$, $0<\beta<1$.
Note that $\widetilde{\varphi}\Phi=\widetilde{\varphi}$.
Hence $\Phi_{\varphi}^*=\Phi_{\widetilde{\varphi}}^*$ if and only if $H_{\beta}(\varphi\|\widetilde{\varphi})= H_{\beta}(\varphi\Phi\|\widetilde{\varphi})$.
Then the equivalent classes can be described by relative quasi-entropies.
\begin{theorem}\label{thm:relative quasi entropy}
    Suppose $\mathcal{N}\subseteq\mathcal{M}$ is an irreducible finite inclusion of factors and $\varphi,\psi$ are faithful normal states, $0<\beta<1$. 
Suppose there exists a bimodule quantum channel $\Phi$ such that $\varphi\Phi=\psi$.
    Then we have that $\varphi$ is equivalent to $\psi$ if and only if 
\begin{align*}
    H_{\beta}(\varphi\|\widetilde{\varphi})= H_{\beta}(\psi\|\widetilde{\varphi}),
\end{align*}
where $\widetilde{\varphi}=\varphi\mathcal{E}_{\mathcal{N}}$.
\end{theorem}
\begin{proof}
    Suppose that $\varphi$ is equivalent to $\psi$.
    Then there also exists a bimodule quantum channel $\Psi$ such that  $\psi\Psi=\varphi$, and by the condition we have $\varphi\Phi=\psi$.
    By the monotonicity of the relative quasi-entropy, we have that
\begin{align*}
    H_{\beta}(\varphi\|\widetilde{\varphi})&=H_{\beta}(\psi\Psi\|\widetilde{\varphi})=H_{\beta}(\psi\Psi\|\widetilde{\varphi}\Psi)\\
    &\geq H_{\beta}(\psi\|\widetilde{\varphi})=H_{\beta} (\varphi\Phi\|\widetilde{\varphi})=H_{\beta} (\varphi\Phi\|\widetilde{\varphi}\Phi)
    \\&\geq H_{\beta} (\varphi\|\widetilde{\varphi}).
\end{align*}
This implies that $H_{\beta}(\varphi\|\widetilde{\varphi})=H_{\beta} (\psi\|\widetilde{\varphi})$.

Suppose that  $H_{\beta}(\varphi\|\widetilde{\varphi})=H_{\beta}(\psi\|\widetilde{\varphi})$.
By the discussion before the theorem, we have that $\Phi_{\varphi}^*=\Phi_{\widetilde{\varphi}}^*$.
Hence
$\varphi\Phi\Phi_{\widetilde{\varphi}}^*= \varphi\Phi \Phi_{\varphi}^*=\varphi$.
By Lemma \ref{lem:adjont bimodule map}, we see that $\Phi_{\widetilde{\varphi}}^*$ is a bimodule quantum channel.
Therefore, we have that
\begin{align*}
    \varphi\Phi=\psi,\quad \psi\Phi_{\widetilde{\varphi}}^*=\varphi.
\end{align*}
This completes the proof.
\end{proof}
\begin{remark}
    In the proof of Theorem \ref{thm:relative quasi entropy}, we need the condition that $\widetilde{\varphi}\Phi=\widetilde{\varphi}$ for any bimodule quantum channel $\Phi$, which is due to  the irreducibility of the subfactor.
\end{remark}
Recall that the Connes cocycle for $\varphi$, $\widetilde{\varphi}$ is defined as follows:
\begin{align*}
    [D\varphi,D\widetilde{\varphi}]_t=\Delta_{\varphi,\widetilde{\varphi}}^{it}\Delta_{\widetilde{\varphi},\widetilde{\varphi}}^{-it}.
\end{align*}
Theorem \ref{thm:relative quasi entropy} can be formulated in terms of Connes cocycles.
\begin{corollary}
    Suppose $\mathcal{N}\subseteq\mathcal{M}$ is an irreducible finite inclusion of factors and $\varphi,\psi$ are faithful normal states.
    We have that $\varphi$ is equivalent to $\psi$ if and only if there exists a bimodule quantum channel $\Phi$ such that $\varphi\Phi=\psi$ and
\begin{align*}
    \Phi([D\varphi\Phi,D\widetilde{\varphi} ]_t)=[D\varphi,D\widetilde{\varphi}]_t,\quad t\in\mathbb{R}
\end{align*}
or
\begin{align*}
    \Phi\Phi_{\widetilde{\varphi}}^*([ D\varphi,D\widetilde{\varphi}]_t)=[D\varphi,D\widetilde{\varphi}]_t,\quad t\in\mathbb{R}.
\end{align*}
\end{corollary}
\begin{proof}
    It follows from Theorem \ref{thm:relative quasi entropy} and \cite[Theorem 3]{Pet88}.
\end{proof}

\subsection{Free energy}\label{sec:free energy and entropy}
In \cite{LX18,LX20,Xu20}, Lango and Xu gave a rigorous explicit computation for relative entropy and entanglement entropy within the framework of quantum field theory  and conformal field theory by utilizing von Neumann algebras, modular theory, and operator theory.
The results connect relative entropy and index of subfactors, which is different but related to Pimsner-Popa result that connects Connes-St\o rmer entropy to index.
The relative entropy has a closed relation to free energy.

Ehrenfest \cite{Ehr33} classified the phase transitions based on the derivatives of the thermodynamic free energy.
A phase transition is first-order if the first derivative is discontinuous and is second-order if the first derivative is continuous.
The Landau free energy is a foundational concept in condensed matter physics and statistical mechanics that provides a powerful, phenomenological framework for understanding phase transitions \cite{Lan37,Dev49,Dev51,Dev54}.
Suppose $\beta=\dfrac{1}{k_BT}$, where $T$ is the temperature and $k_B$ is the Boltzman’s constant.
Suppose $H$ is the Hamiltonian of the system.
The Landau free energy $F$ is defined to be
\begin{align*}
    F=-\beta^{-1}\log{\rm Tr(e^{-\beta H}}).
\end{align*}

In this section, we will discuss the free energy in a subfactor and its relation with the index of the subfactor.
 Suppose $\mathcal{N}\subseteq\mathcal{M}$ is an irreducible finite inclusion of factors and $\varphi$ is a faithful normal state on $\mathcal{M}$.
 We recall that $\widetilde{\varphi}=\varphi\mathcal{E}_{\mathcal{N}}$,  $\Delta_{\widetilde{\varphi},\varphi}$ is the relative modular operator of  $\widetilde {\varphi}$ and $\varphi$, and $H_{\varphi}=-\log \Delta_{\widetilde{\varphi},\varphi}$ is the modular Hamiltonian.
The free energy $F_{\varphi}$ is defined as follows:
\begin{align*}
    F_{\varphi}&=-\beta^{-1}\log\langle e^{-\beta H_{\varphi}} \Omega_{\varphi},\Omega_{\varphi}\rangle\\
&=-\beta^{-1}\log\langle\Delta_{\widetilde{\varphi},\varphi}^{\beta}\Omega_{\varphi},\Omega_{\varphi}\rangle.
\end{align*}

The Araki relative entropy $H(\varphi\|\widetilde{\varphi})$ \cite{Ara75,Ara77} is defined as 
\begin{align*}
    H(\varphi\|\widetilde{\varphi})=-\langle \log \Delta_{\widetilde{\varphi},\varphi}\Omega_{\varphi},\Omega_{\varphi}\rangle.
\end{align*}
We have that $H(\varphi\|\widetilde{\varphi})\geq 0$.
One can obtain the relation between the free energy and relative entropy by differentiating the free energy with respect to $\beta$:
\begin{align*}
    \frac{dF_{\varphi}}{d\beta}=\beta^{-2}\log\langle e^{-\beta H_{\varphi}}\Omega_{\varphi},\Omega_{\varphi}\rangle-\beta^{-1} \frac{\langle\Delta_{\widetilde{\varphi},\varphi}^{\beta}\log\Delta_{\widetilde{\varphi},\varphi} \Omega_{\varphi},\Omega_{\varphi}\rangle}{\langle\Delta_{\widetilde{\varphi},\varphi}^{\beta} \Omega_{\varphi},\Omega_{\varphi}\rangle}.
\end{align*}
Evaluating it at $\beta=1$, we see that
\begin{align*}
    \frac{dF_{\varphi}}{d\beta}\bigg|_{\beta=1}=H(\varphi\|\widetilde{\varphi}).
\end{align*}
Let $\tau$ be the normal faithful trace on $\mathcal{M}$.
Suppose $\varphi(x)=\tau(h_{\varphi}x)$ and $\widetilde{\varphi}(x)=\tau(h_{\widetilde{\varphi}}x)$, where $h_{\varphi},h_{\widetilde{\varphi}}\in L^1(\mathcal{M})^+$ are density operators.
Then the relative entropy is equal to
\begin{align*}
    H(\varphi\|\widetilde{\varphi})=\tau(h_{\varphi}\log h_{\varphi}-h_{\varphi}\log h_{\widetilde{\varphi}}).
\end{align*}
By \cite{OP93} or \cite[Theorem 2.1 (3)]{Xu20} and Pimser-Popa inequality,    $H(\varphi\|\widetilde{\varphi})\leq \log[\mathcal{M}:\mathcal{N}]$.
Suppose $\mathcal{N}_{-1}\subseteq\mathcal{N}\subseteq^{e_{-1}}\mathcal{M}$ is the downward basic construction.
Let $\varphi(x)=[\mathcal{M}:\mathcal{N}]\tau(xe_{-1})$.
Then the inequality becomes equality.
So we have
\begin{align*}
    \sup_{\varphi\in\mathcal{S}(\mathcal{M})}H(\varphi\|\widetilde{\varphi})=\log[\mathcal{M}:\mathcal{N}],
\end{align*}
where $\mathcal{S}(\mathcal{M})$ is the space of normal states on $\mathcal{M}$.
We also refer to \cite[Theorem 3.1]{GJL20} for the sandwiched R\'{e}nyi relative entropy in the above equation.
If $\varphi$ is a symmetry-protected state with respect to the symmetry $\mathcal{P}$, then 
\begin{align*}
    H(\varphi\|\widetilde{\varphi})\leq \sup_{\psi\in\mathcal{S}(\mathcal{P})}H(\psi\|\widetilde{\psi})= \log[\mathcal{P}:\mathcal{N}].
\end{align*}

Combining the upper bound for the relative entropy, we see that the free energy has following close relation to the Jones index.
\begin{proposition}\label{prop:free energy}
    If $\varphi$ is a $\mathcal{P}$-protected state, we have that
\begin{align*}
    0\leq  \frac{dF_{\varphi}}{d\beta}\bigg|_{\beta=1}\leq\log[\mathcal{P}:\mathcal{N}].
\end{align*}
\end{proposition}

\section{Operator Algebra Symmetries of states}\label{sec:symmetries}
In this section,  we introduce the operator algebra (OA) symmetry of a state, which generalizes the group symmetry in Landau phase transition theory.
We define an equivalence relation among OA symmetries via bishifts of biprojections. 
Based on this equivalence relation, we obtain a comparison of projections equipped with $\mathfrak{F}$-positivity.

\subsection{Operator algebra symmetry}
Suppose $\mathcal{N}\subseteq\mathcal{M}$ is an irreducible finite inclusion of factors.
The subalgebra $\mathcal{N}$ can be viewed as a symmetry of the quantum system $\mathcal{N}\subseteq\mathcal{M}$. 
It is then reasonable to consider the subalgebras of $\mathcal{M}$ containing $\mathcal{N}$ as a symmetry of $\mathcal{N}\subseteq\mathcal{M}$.
Suppose $\mathcal{P}$ is an intermediate von Neumann algebra for  $\mathcal{N}\subseteq\mathcal{M}$. 
Note that $\mathcal{N}'\cap\mathcal{M}=\mathbb{C}\mathbf{1}$ and $\mathcal{N},\mathcal{M}$ are factors.
We see that $\mathcal{P}$ is an intermediate subfactor.
We denote by $\mathcal{E}_{\mathcal{P}}$ the  conditional expectation from $\mathcal{M}$ onto $\mathcal{P}$.

Thanks to Jones subfactor theory and quantum Fourier analysis, we have the following definition for the symmetry of $\mathcal{N}\subseteq\mathcal{M}$.
\begin{definition}(Symmetries)
    Suppose $\mathcal{N}\subseteq\mathcal{M}$ is an irreducible finite inclusion of factors.
    We view the intermediate subfactors as the {\sl symmetries} for $\mathcal{N}\subseteq\mathcal{M}$.
    Suppose $P$ is the biprojection associated to the intermediate subfactor $\mathcal{P}$.
    We equivalently view $P$ as a symmetry.
    Suppose $\mathcal{Q}$ is another symmetry.
    We say $\mathcal{P}$ is a {\sl subsymmetry} of $\mathcal{Q}$ if $\mathcal{P}\subseteq\mathcal{Q}$.
\end{definition}
Next we introduce a relation between two symmetries.
\begin{definition}
    Suppose $\mathcal{N}\subseteq\mathcal{M}$ is an irreducible finite inclusion of factors.
    Suppose $P,Q$ are two biprojections.
    We say $P,Q$ are equivalent if there exists a bishift of biprojection $V$ such that $V^*V=P$ and $VV^*=Q$, denoted by $P\sim Q$.
    In this case, $P$ and $Q$ are {\sl equivalent symmetries}.
    We say $P$ is a {\sl lower symmetry} of $Q$ if $P$ is equivalent to a subsymmetry of $Q$, denoted by $P\preceq Q$.
    We use $P\prec Q$ to denote $P\preceq Q$ and $P$ is not equivalent to $Q$.
\end{definition}
\begin{proposition}\label{prop:partial order}
    The relation ``$\preceq$" is a partial order.
    That is
    \begin{enumerate}[(i)]
        \item $P\preceq P$,
        \item $P\preceq Q$ and $Q\preceq P$ implies $P\sim Q$,
        \item $P\preceq Q$ and $Q\preceq R$ implies $P\preceq R$,
    \end{enumerate}
    where $P,Q,R$ are symmetries.
\end{proposition}
\begin{proof}
    (i) is clear.
    (ii) Suppose $V$ is a  bishift of biprojection such that $V^*V=P, VV^*\leq Q$.
    By the condition, $tr_2(P)=tr_2(Q)$.
    Hence, $VV^*=Q$  by the faithfulness of $tr_2$.
    We have $P$ is equivalent to $Q$.
    (iii) Suppose $V,W$ are bishifts of biprojections such that $V^*V=P, VV^*\leq Q$ and $W^*W=Q,WW^*\leq R$.
Then $WV$ is a partial isometry such that $(WV)^*WV=P$ and $WV(WV)^*\leq R$.
By quantum Hausdorff-Young inequality and quantum Young inequality in \cite{JLW16} and Theorem \ref{thm:bishift}, we have
\begin{align*}
   \delta= \delta\|WV\|_{\infty}\leq\|\mathfrak{F}^{-1}(WV)\|_1&=\|\mathfrak{F}^{-1}(V)\ast\mathfrak{F}^{-1}(W)\|_1\\
   &\leq \frac{\|\mathfrak{F}^{-1}(V)\|_1\|\mathfrak{F}^{-1}(W)\|_1}{\delta}=\delta\|V\|_{\infty}\|W\|_{\infty}=\delta.
\end{align*}
Hence $\mathfrak{F}^{-1}(WV)$ is extremal.
Thus $WV$ is a bishift of a biprojection by Theorem \ref{thm:bishift}.
Moreover, $WV(WV)^*\leq R$ is a biprojection.
    Therefore, $P\preceq R$.
\end{proof}
\begin{remark}
  Note that  an $\mathfrak{F}$-positive partial isometry is a bishift of biprojection and an $\mathfrak{F}$-positive projection is a biprojection.
  Hence, Proposition \ref{prop:partial order} provides a comparison of projections equipped with $\mathfrak{F}$-positivity.
\end{remark}

With the symmetries of the quantum system, we introduce the states preserved by the symmetries.
\begin{definition}\label{def:Symmetry Protected States}(Symmetry-Protected States)
    Suppose $\mathcal{N}\subseteq\mathcal{M}$ is an irreducible finite inclusion of factors, $\varphi$ is a normal state on $\mathcal{M}$.
    Suppose $\mathcal{P}$ is an intermediate subfactor.
    We say $\varphi$ is {\sl protected} by the symmetry $\mathcal{P}$ if $\varphi\mathcal{E}_{\mathcal{P}}=\varphi$.
\end{definition}
Next we will introduce the symmetries of states.
Suppose $\mathcal{M}$ is a finite von Neumann algebra with a normal faithful trace $\tau$.
Let $\varphi$ be a normal state on $\mathcal{M}$.
We first recall that the symmetry of $\varphi$ is the following von Neumann algebra:
\begin{align*}
    \mathcal{A}_{\varphi}:=\{u\in\mathcal{M}: \varphi(u^*xu)=\varphi(x), x\in\mathcal{M} \}''.
\end{align*}
In particular, when $\varphi=\tau$, then $\mathcal{A}_{\tau}=\mathcal{M}$.
The density operator $h_{\varphi}$ is given by $\varphi(x)=\tau(h_{\varphi}x)$, $x\in\mathcal{M}$, which is an element in $L^1(\mathcal{M})$.
We have that 
\begin{align*}
    \mathcal{A}_{\varphi}=\{h_{\varphi}\}'\cap \mathcal{M}.
\end{align*}
Let $\mathcal{E}_{\mathcal{A}_{\varphi}}$ be the trace-preserving conditional expectation from $\mathcal{M}$ onto $\mathcal{A}_{\varphi}$.
Then 
\begin{align*}
    \varphi=\varphi\mathcal{E}_{\mathcal{A}_{\varphi}}.
\end{align*}
Using this equality, we can generalize the symmetry to subfactors.
Suppose $\mathcal{N}\subseteq\mathcal{M}$ is a finite inclusion of  factors.
If $\varphi|_{\mathcal{N}}=\tau|_{\mathcal{N}}$, then the symmetry $\mathcal{A}_{\varphi}$ is an intermediate subfactor of $\mathcal{N}\subseteq\mathcal{M}$.
We introduce the following definition. 
\begin{definition}(Operator Algebra Symmetries of States)\label{def:OA symmetry}
Suppose $\mathcal{N}\subseteq\mathcal{M}$ is an irreducible finite inclusion of factors, $\varphi$ is a normal state on $\mathcal{M}$.
Let $\mathcal{N}\subseteq\mathcal{P}_{\varphi}\subseteq\mathcal{M}$ be the minimal intermediate subfactor such that $\varphi$ is protected by $\mathcal{P}_{\varphi}$.
We call $\mathcal{P}_{\varphi}$ the {\sl operator algebra (OA) symmetry} of the state $\varphi$.
\end{definition}
\begin{remark}
    Suppose $\{\mathcal{P}_{i}\}_{i\in I}$ is the set of all symmetries  such that $\varphi\mathcal{E}_{\mathcal{P}_i}=\varphi$.
    Then the symmetry of $\varphi$ is given by $\mathcal{P}_{\varphi}=\bigcap_{i\in I}\mathcal{P}_{i}$.
\end{remark}
\begin{remark}\label{rem:group symmetry}
    Symmetries of states generalize the group symmetries.
Let $G$ be a finite group with an out action $\alpha$ on $\mathcal{M}$.
Let $\mathcal{N}=\mathcal{M}^G$ be the fixed points algebra.
Then $\mathcal{N}\subseteq\mathcal{M}$ is an irreducible subfactor with index $|G|$.
For any $g\in G$ and normal state $\varphi$, define $g\varphi(x)=\varphi(\alpha_g(x))$, $x\in\mathcal{M}$.
Let $H<G$ be a subgroup.
Then $g\varphi=\varphi$, $\forall g\in H$, if and only if $\varphi \mathcal{E}_{\mathcal{M}^H}=\varphi$, where $\mathcal{M}^H$ is the  symmetry of $\varphi$.

\end{remark}

\begin{remark}
  Noether's theorem states that  every continuous symmetry of the action of a physical system corresponds to a conserved quantity.
  In this paper, we use a bimodule quantum channel to represent the action and consider the algebra generated by conserved quantities, which can be regarded as the generalization of Noether's theorem.
\end{remark}

In terms of Hamiltonians, we see that the modular Hamiltonian is invariant under the corresponding symmetries.
\begin{proposition}
     Suppose $\mathcal{N}\subseteq\mathcal{M}$ is an irreducible finite inclusion of factors, $\varphi$ is the ground state on $\mathcal{M}$ protected by a symmetry $P$.
     Then 
     \begin{align*}
         PH_{\varphi}\subset H_{\varphi}P,
     \end{align*}
i.e. the modular Hamiltonian $H_{\varphi}$ is invariant under the symmetry $P$.
\end{proposition}
\begin{proof}
    Let $\mathcal{P}$ be the intermediate subfactor associated to the biprojection $P$ and $\mathcal{E}_{\mathcal{P}}$ the conditional expectation onto $\mathcal{P}$.
    Then for any $x\in\mathcal{M}$, we have that
    \begin{align*}
        S_{\widetilde{\varphi},\varphi}Px\Omega_{\varphi}&=\mathcal{E}_{\mathcal{P}}(x^*)\Omega_{\widetilde{\varphi}}\\
        &=P S_{\widetilde{\varphi},\varphi}x\Omega_{\varphi},
    \end{align*}
    where $\mathcal{E}_{\mathcal{P}}(x)\Omega_{\varphi}=Px\Omega_{\varphi}$ and $S_{\widetilde{\varphi},\varphi}x\Omega_{\varphi}=x^*\Omega_{\widetilde{\varphi}}$.
    We then obtain that $PS_{\widetilde{\varphi},\varphi}\subset S_{\widetilde{\varphi},\varphi}P$ and $PS_{\widetilde{\varphi},\varphi}^*\subset S_{\widetilde{\varphi},\varphi}^*P$.
    Hence $P\Delta_{\widetilde{\varphi},\varphi}\subset\Delta_{\widetilde{\varphi},\varphi}P$. 
    Finally, we see that $PH_{\varphi}\subset H_{\varphi}P$.
\end{proof}

\section{Finiteness of equivalent states in a phase}\label{sec:finiteness of phases}
In classical phase transition, the phase usually contains infinitely many states due to conjugations of unitaries.
In our settings, when the inclusion $\mathcal{N}\subseteq\mathcal{M}$ is an irreducible finite inclusion of factors, we discover a new phenomenon that the number of states in a phase is finite and bounded by an index-dependent constant.

Let $\varphi$ be a normal state on $\mathcal{M}$.
Let $[\varphi]$ be the set of all normal states that are equivalent to $\varphi$.
The following theorem is the main result in this section.
\begin{theorem}\label{thm:phase finite upper bound}
    Suppose $\mathcal{N}\subseteq\mathcal{M}$ is an irreducible finite inclusion of factors and $\varphi$ is a normal state.
    Then $\#[\varphi]\leq [\mathcal{M}:\mathcal{N}]2^{\lfloor [\mathcal{M}:\mathcal{N}] \rfloor-1}$.
\end{theorem}
To prove Theorem \ref{thm:phase finite upper bound}, we need some preparations.
First, we recall some results for bimodule quantum channels and bishit of biprojections in \cite{HJLW25,JLW16} as the followings.
Theorem \ref{thm:bishift} and Proposition \ref{prop:portion} will be applied in the proof of Theorem \ref{thm:state equivalent}.
\begin{proposition}[Corollary 6.12 in \cite{JLW16}]\label{prop:portion}
    Suppose that $\mathcal{N}\subseteq\mathcal{M}$ is an irreducible finite inclusion of factors and $\Phi$ is a bimodule quantum channel.
Then the spectral projection $Q$ of $|\Phi|$ with spectrum $\|\Phi\|_{\infty}=1$ is a biprojection and $\Phi Q$ is a bishift of biprojection.
We denote $\Phi Q$ by $V_{\Phi}$ and call it as the spectral portion of $\Phi$.
\end{proposition}
\begin{remark}
    The spectral portion $V_{\Phi}$ of $\Phi$ is $\mathfrak{F}$-positive by the quantum Schur product theorem:
\begin{align*}
    \mathfrak{F}(V_{\Phi})=\mathfrak{F}(\Phi)\ast\mathfrak{F}(Q)\geq 0.
\end{align*}
Moreover, $\|\mathfrak{F}(V_{\Phi})\|_1=[\mathcal{M}:\mathcal{N}]^{1/2}$.
\end{remark}
Suppose  $Q$ is the spectral projection of $|\Phi|$ with spectrum $\|\Phi\|_{\infty}=1$ and $P$ is the spectral projection of $|\Phi^*|$ with spectrum $\|\Phi^*\|_{\infty}=1$ respectively.
Then $P,Q$ are given by Equation \eqref{eq:support}.
Thus, by Equation \eqref{eq:relation between two conditional expectation},
\begin{align}\label{eq:two relation between two spectral projection}
   P=\Phi Q\Phi^*,\quad Q=\Phi^*P\Phi.
\end{align}
\begin{remark}\label{rem:characterization of bishif of biprojection}
    By Theorem \ref{thm:bishift} and the definition of extremal elements, we have that if $V$ is a partial isometry and $\|\mathfrak{F}(V)\|_1=[\mathcal{M}:\mathcal{N}]^{1/2}$, then $V$ is a bishift of biprojection.
\end{remark}

Next we assume that $\varphi$ is equivalent to $\psi$ under bimodule quantum channels.
Then we have
\begin{align*}
        \varphi\Phi=\psi,\quad \psi \Psi=\varphi,
    \end{align*}
where $\Phi$ and $\Psi$ are bimodule quantum channels.
Let $\mathcal{P}_{\varphi}$ and $\mathcal{P}_{\psi}$ be the OA symmetries of $\varphi$ and $\psi$ respectively  as given in Definition \ref{def:OA symmetry}.
It follows that
\begin{align*}
      \varphi \mathcal{E}_{\mathcal{P}_{\varphi}}\Phi\mathcal{E}_{\mathcal{P}_{\psi}}=\psi,\quad \psi \mathcal{E}_{\mathcal{P}_{\psi}}\Psi\mathcal{E}_{\mathcal{P}_{\varphi}}=\varphi.
\end{align*}
For simplicity, we set $\Phi:=\mathcal{E}_{\mathcal{P}_{\varphi}}\Phi\mathcal{E}_{\mathcal{P}_{\psi}}$ and $\Psi:=\mathcal{E}_{\mathcal{P}_{\psi}}\Psi\mathcal{E}_{\mathcal{P}_{\varphi}}$.
    Then we have that
\begin{align*}
    \varphi\Phi\Psi=\varphi,\quad \psi\Psi\Phi=\psi.
\end{align*}
By the minimality of $\mathcal{P}_{\varphi}$ and $\mathcal{P}_{\psi}$, and Proposition \ref{prop:conditional expectation}, we have
\begin{align}\label{eq:conditional expectation}
    \mathcal{E}_{\mathcal{P}_{\varphi}}=\lim_{n\to \infty}\frac{1}{n}\sum_{k=1}^n(\Phi\Psi)^k,\quad   \mathcal{E}_{\mathcal{P}_{\psi}}=\lim_{n\to \infty}\frac{1}{n}\sum_{k=1}^n(\Psi\Phi)^k.
\end{align}
It follows that
\begin{align}\label{eq:relation conditional expectation}
      \mathcal{E}_{\mathcal{P}_{\varphi}}=\Phi  \mathcal{E}_{\mathcal{P}_{\psi}} \Psi,\quad \mathcal{E}_{\mathcal{P}_{\psi}}=\Psi \mathcal{E}_{\mathcal{P}_{\varphi}}\Phi. 
\end{align}
Now we are able to characterize the equivalence between two states by bishifts of biprojections between two symmetries. 
\begin{theorem}\label{thm:state equivalent}
    Suppose $\mathcal{N}\subseteq\mathcal{M}$ is an irreducible finite inclusion of factors and $\varphi,\psi$ are normal states.
    Then $\varphi$ is equivalent to $\psi$ if and only if there exists a bishift of biprojection $V$ such that
\begin{align*}
    \varphi \Phi_V=\psi,\quad \psi \Phi_{V^*}=\varphi,
\end{align*}
where $\Phi_V$ and $\Phi_{V^*}$ are the bimodule quantum channels associated with $V$ and $V^*$ respectively.
Moreover,
\begin{align*}
    VV^*=P_{\varphi},\quad V^*V=P_{\psi},
\end{align*}
where $P_{\varphi}, P_{\psi}$ are biprojections corresponding to intermediate subfactors $\mathcal{P}_{\varphi},\mathcal{P}_{\psi}$.
\end{theorem}
\begin{proof}
The sufficiency is clear and we show the necessity.
    Suppose that $\varphi$ is equivalent to $\psi$.
By Equation \eqref{eq:relation conditional expectation}, 
\begin{align*}
      P_{\varphi}=\Phi  P_{\psi} \Psi,\quad P_{\psi}=\Psi P_{\varphi}\Phi. 
\end{align*}
By \cite[Theorem 4.4]{HJLW25}, $P_{\varphi}, P_{\psi}$ are the spectral projections of $\Phi\Psi$, $\Psi\Phi$ corresponding to $1$ respectively.
That is
\begin{align*}
    P_{\varphi}\Phi\Psi P_{\varphi}=P_{\varphi},\quad P_{\psi}\Psi\Phi P_{\psi}=P_{\psi}.
\end{align*}
Note that
\begin{align*}
    P_{\varphi}=P_{\varphi}\Phi\Psi P_{\varphi}\Psi^*\Phi^*P_{\varphi}\leq P_{\varphi}\Phi\Phi^*P_{\varphi}\leq P_{\varphi},\quad P_{\varphi}=P_{\varphi}\Psi^*\Phi^*P_{\varphi}\Phi\Psi P_{\varphi}\leq P_{\varphi}\Psi^*\Psi P_{\varphi}\leq P_{\varphi}.
\end{align*}
Hence $ P_{\varphi}\Phi\Phi^*P_{\varphi}= P_{\varphi}$ and $P_{\varphi}\Psi^*\Psi P_{\varphi}=P_{\varphi}$.
Therefore,
\begin{align*}
    (P_{\varphi}\Phi-P_{\varphi}\Psi^*)(P_{\varphi}\Phi-P_{\varphi}\Psi^*)^*=P_{\varphi}\Phi\Phi^*P_{\varphi}-P_{\varphi}\Phi\Psi P_{\varphi}-P_{\varphi}\Psi^*\Phi^*P_{\varphi}+P_{\varphi}\Psi^*\Psi P_{\varphi}=0.
\end{align*}
We obtain that $P_{\varphi}\Phi=P_{\varphi}\Psi^*(=:V)$ is a partial isometry.
Then $VV^*=P_{\varphi}$ and $V^*V=\Psi P_{\varphi}\Phi =P_{\psi}$. 

Next we prove that $V$ is a bishift of biprojection such that $\mathfrak{F}(V)$ is positive and $\|\mathfrak{F}(V)\|_1=[\mathcal{M}:\mathcal{N}]^{1/2}$.
Let $V_{\Phi}$ be the spectral portions of $\Phi$.
By Equation \eqref{eq:two relation between two spectral projection}, one can check that $(P_{\varphi}V_{\Phi}-V)(P_{\varphi}V_{\Phi}-V)^*=0$  and then $P_{\varphi}V_{\Phi}=V$.
Thus 
\begin{align*}
    \mathfrak{F}(V)=\mathfrak{F}(P_{\varphi})\ast \mathfrak{F}(V_{\Phi})\geq 0.
\end{align*}
Moreover,
\begin{align*}
    \|\mathfrak{F}(V)\|_1=\|\mathfrak{F}(P_{\varphi})\ast \mathfrak{F}(V_{\Phi})\|_1=\frac{\|\mathfrak{F}(P_{\varphi})\|_1 \|\mathfrak{F}(V_{\Phi})\|_1}{[\mathcal{M}:\mathcal{N}]^{1/2}}=[\mathcal{M}:\mathcal{N}]^{1/2}.
\end{align*}
Thus, 
\begin{align*}
    \|V\|_{\infty}=1=\frac{ \|\mathfrak{F}(V)\|_1}{[\mathcal{M}:\mathcal{N}]^{1/2}},
\end{align*}
which follows that $\mathfrak{F}^{-1}(V)$ is extremal.
Then by Theorem \ref{thm:bishift}, $V$ is a bishift of biprojection.
We have that $\Phi_V$ is a  bimodule quantum channel.
Note that $\Phi_V=\mathcal{E}_{\mathcal{P}_{\varphi}}\Phi$ since $V=P_{\varphi}\Phi$.
We have that $\varphi\Phi_V=\psi$.
Similarly, we have that $\psi\Phi_{V^*}=\varphi$.
This completes the proof.
\end{proof}
\begin{remark}
    Suppose $\varphi$ is equivalent to $\psi$.
    In the proof of Theorem \ref{thm:state equivalent}, we see that the bishift of biprojection $V=\mathcal{P}_{\varphi}\Phi\mathcal{P}_{\psi}=\mathcal{P}_{\varphi}  \Psi^*\mathcal{P}_{\psi}$.
\end{remark}
\begin{corollary}
    Suppose $\mathcal{N}\subseteq\mathcal{M}$ is an irreducible finite inclusion of factors and $\varphi,\psi$ are normal states. 
    We assume that there exists a bimodule quantum channel $\Phi$ such that $\varphi\Phi=\psi$.
    Then $\varphi$ is equivalent to $\psi$ if and only if $\mathcal{P}_{\varphi}\sim\mathcal{P}_{\psi}$, and $\mathcal{P}_{\varphi}\Phi\mathcal{P}_{\psi}$ is a bishift of biprojection mapping $\mathcal{P}_{\psi}$ onto $\mathcal{P}_{\varphi}$.
\end{corollary}
\begin{proof}
Suppose $\varphi$ is equivalent to $\psi$.
By Theorem \ref{thm:state equivalent}, the conclusions hold.
On the contrary, let $V:=\mathcal{P}_{\varphi}\Phi\mathcal{P}_{\psi}$.
    Then we have that  $\Phi_V=\mathcal{E}_{\mathcal{P}_{\varphi}}\Phi\mathcal{E}_{\mathcal{P}_{\psi}}$ and
    \begin{align*}
        \varphi\Phi_V=\varphi\Phi\mathcal{E}_{\mathcal{P}_{\psi}}=\psi\mathcal{E}_{\mathcal{P}_{\psi}}=\psi.
    \end{align*}
    Hence 
    \begin{align*}
        \psi \Phi_{V^*}=\varphi\Phi_{V}\Phi_{V^*}=\varphi\mathcal{E}_{\mathcal{P}_{\varphi}}=\varphi.
    \end{align*}
    We see that $\varphi$ is equivalent to $\psi$.
\end{proof}
By Theorem \ref{thm:state equivalent}, we have the following relation between the cardinality of the phase and the number of bishifts of biprojection.
\begin{corollary}\label{cor:number of phase}
     Suppose $\mathcal{N}\subseteq\mathcal{M}$ is an irreducible finite inclusion of factors and $\varphi$ is a normal state.
     Let $P_{\varphi}$ be the  biprojection corresponding to the  intermediate subfactor $\mathcal{P}_{\varphi}$.
     Define
\begin{align*}
    \mathfrak{S}_1:=&\{V \text{ is a bishift of biprojection in}\ \mathcal{N}'\cap\mathcal{M}_1:\ V\ \text{is $\mathfrak{F}$-positive}, \|\mathfrak{F}(V)\|_1=[\mathcal{M}:\mathcal{N}]^{1/2},\\&\ V^*V=P_{\varphi}, VV^*=Q \ \text{for some biprojection $Q$ in $\mathcal{N}'\cap\mathcal{M}_1$}\},\\
    \mathfrak{S}_2:=&\{V \text{ is a bishift of biprojection in}\ \mathcal{N}'\cap\mathcal{M}_1:\ V\ \text{is $\mathfrak{F}$-positive}, \|\mathfrak{F}(V)\|_1=[\mathcal{M}:\mathcal{N}]^{1/2},\\&\quad  V^*V=P_{\varphi}, VV^*=P_{\varphi}\},\\
    \mathfrak{S}_3:=&\{Q\ \text{is a biprojection in $\mathcal{N}'\cap\mathcal{M}_1$}: Q\sim P_{\varphi} \},
\end{align*}
 where $Q\sim P_{\varphi}$ means there exists a bishift of biprojection $W$ such that $W^*W=Q$ and $WW^*=P_{\varphi}$.
Then we have
\begin{align*}
     \#[\varphi]\leq \#\mathfrak{S}_1= \#\mathfrak{S}_2\times \#\mathfrak{S}_3.
\end{align*}
\end{corollary}
\begin{proof}
    The  inequality follows from Theorem \ref{thm:state equivalent}.
    Next we show the equality.
   Fixed a biprojection $Q\in\mathcal{N}'\cap\mathcal{M}_1$, define the subset $\mathfrak{S}_1^Q$ of $\mathfrak{S}_1$ as
   \begin{align*}
       \mathfrak{S}_1^Q:=&\{V \text{ is a bishift of biprojection in}\ \mathcal{N}'\cap\mathcal{M}_1:\ V\ \text{is $\mathfrak{F}$-positive}, \|\mathfrak{F}(V)\|_1=[\mathcal{M}:\mathcal{N}]^{1/2},\\&\ V^*V=P_{\varphi}, VV^*=Q\}.
   \end{align*}
   For any $V,W\in\mathfrak{S}_1^Q$, we have $W^*V\in\mathfrak{S}_2$.
   Indeed,
   \begin{align*}
      V^*W W^*V=V^*QV=P_{\varphi},\quad W^*VV^*W=W^*QW=P_{\varphi}.
   \end{align*}
   Since $W$ is $\mathfrak{F}$-positive, we have that $W^*$ is $\mathfrak{F}$-positive and $tr_2(\mathfrak{F}(W^*))=tr_2(\mathfrak{F}(W))=[\mathcal{M}:\mathcal{N}]^{1/2}$.
   Thus by the quantum Schur product theorem, $W^*V$ is $\mathfrak{F}$-positive and 
   \begin{align*}
       \|\mathfrak{F}(W^*V)\|_1=[\mathcal{M}:\mathcal{N}]^{1/2}.
   \end{align*}
   By Theorem \ref{thm:bishift} and Remark \ref{rem:characterization of bishif of biprojection}, $W^*V$ is a bishift of biprojection.
   Thus $W^*V\in\mathfrak{S}_2$.
   Now suppose $V_1,V_2,W\in\mathfrak{S}_1^Q$ such that $W^*V_1=W^*V_2$.
   Then $V_1=QV_1=WW^*V_1=WW^*V_2=QV_2=V_2$.
   It follows that for a fixed element $W\in\mathfrak{S}_1^Q$, the following map induced by $W$
   \begin{align*}
       \mathfrak{S}_1^Q\to \mathfrak{S}_2,\quad V\mapsto W^*V
   \end{align*}
   is injective.
   On the other hand, for any $U\in\mathfrak{S}_2$, we can similarly prove that $WU\in\mathfrak{S}_1^Q$.
   Hence the map is also surjective.
   Thus $\#\mathfrak{S}_1^Q=\#\mathfrak{S}_2$ for any biprojection $Q\sim P_{\varphi}$.
   We have
   \begin{align*}
       \#\mathfrak{S}_1=\sum_{Q\sim P_{\varphi}}\#\mathfrak{S}_1^Q=\sum_{Q\sim P_{\varphi}}\#\mathfrak{S}_2=\#\mathfrak{S}_2\times \#\mathfrak{S}_3.
   \end{align*}
   The proof is completed.
\end{proof}
Next, we will estimate $\#\mathfrak{S}_2$ and $\#\mathfrak{S}_3$ by Jones index.
Before proving the following lemma, we recall some existing results of bishifts of biprojections.
Suppose $\mathcal{N}\subseteq\mathcal{M}$ is an irreducible finite inclusion of factors.
Suppose $V$ is a bishift of biprojection in $\mathcal{N}'\cap\mathcal{M}_1$.
By \cite[Main Theorem 2]{JLW16}, $V$ is a minimizer of the quantum Donoho-Stark uncertainty principle:
\begin{align*}
    \mathcal{S}(V)\mathcal{S}(\mathfrak{F}(V))=[\mathcal{M}:\mathcal{N}].
\end{align*}
So if $V$ is a unitary, then $\mathcal{S}(V)=[\mathcal{M}:\mathcal{N}]$, which implies that $\mathcal{S}(\mathfrak{F}(V))=1$.
Hence the range projection $\mathcal{R}(\mathfrak{F}(V))$ is a trace-one projection.
Thus $\mathfrak{F}(V)$ is a multiple of a trace-one projection.

\begin{lemma}\label{lem:bishift number}
    Suppose $\mathcal{N}\subseteq\mathcal{M}$ is an irreducible finite inclusion of factors and $P$ is a biprojection in $\mathcal{N}'\cap\mathcal{M}_1$.
Then
\begin{align*}
    &\#\{V \text{ is a bishift of biprojection in}\ \mathcal{N}'\cap\mathcal{M}_1:\ V\ \text{is $\mathfrak{F}$-positive}, \|\mathfrak{F}(V)\|_1=[\mathcal{M}:\mathcal{N}]^{1/2},\\&\quad  V^*V=P, VV^*=P\}\leq [\mathcal{P}:\mathcal{N}]
\end{align*}
where $\mathcal{P}$ is the intermediate subfactor associated to the biprojection $P$.
In particular,
\begin{align*}
    \#\mathfrak{S}_2\leq [\mathcal{P}_{\varphi}:\mathcal{N}],
\end{align*}
where $\mathcal{P}_{\varphi}$ is the intermediate subfactor associated to the biprojection $P_{\varphi}$.
\end{lemma}
\begin{proof}
    Let $\mathscr{P}^{\mathcal{N}\subseteq\mathcal{M}}$ be the irreducible subfactor planar algebra associated to $\mathcal{N}\subseteq\mathcal{M}$.
    Then the planar algebra $\mathscr{P}^{\mathcal{N}\subseteq\mathcal{P}}$ associated to $\mathcal{N}\subseteq\mathcal{P}$ is obtained by adding biprojection $P$ in the following way:
    \begin{align*}
        \raisebox{-0.5cm}
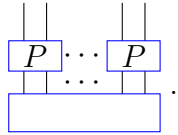
{
\begin{tikzpicture}
\begin{scope}[shift={(0,0.8)}]
\draw (0.2,0.9)--(0.2,-0.4);
\draw (0.5,0.9)--(0.5,-0.4);
\draw [blue, fill=white] (0,0) rectangle (0.7,0.4);
\node at (0.35,0.2) {$P$};
\end{scope}
\begin{scope}[shift={(1.3,0.8)}]
\draw (0.2,0.9)--(0.2,-0.4);
\draw (0.5,0.9)--(0.5,-0.4);
\draw [blue, fill=white] (0,0) rectangle (0.7,0.4);
\node at (0.35,0.2) {$P$};
\end{scope}
\draw [blue, fill=white] (0,0) rectangle (2,0.5);
\node at (1,1) {$\cdots$};
\node at (1,0.65) {$\cdots$};
\end{tikzpicture}}\;.
    \end{align*}
Hence for each bishift of biprojection $V$ in $\mathcal{N}'\cap\mathcal{M}_1$ is unitary in  $\mathcal{N}'\cap\mathcal{P}_1$.
By the above discussion, we see that $\mathfrak{F}_{\mathcal{N}\subseteq\mathcal{P}}(V)$  is the multipication of $[\mathcal{M}:\mathcal{N}]^{1/2}$ and a trace-one projection (equivalently, group-like projection) in $\mathcal{P}'\cap \mathcal{P}_2$.
Therefore, 
\begin{align*}
    &\#\{V \text{ is a bishift of biprojection in}\ \mathcal{N}'\cap\mathcal{M}_1:\ V\ \text{is $\mathfrak{F}$-positive}, \|\mathfrak{F}(V)\|_1=[\mathcal{M}:\mathcal{N}]^{1/2},\\&\quad  V^*V=P, VV^*=P\}\leq\# \{\text{trace-one projections in $\mathcal{P}'\cap \mathcal{P}_2$}\}  \leq [\mathcal{P}:\mathcal{N}].
\end{align*}
We obtain the conclusion.
\end{proof}
Now we are able to prove the main result in this section.
\begin{proof}[Proof of Theorem \ref{thm:phase finite upper bound}]
By Lemma \ref{lem:bishift number}, we have that 
$$ \#\mathfrak{S}_2\leq [\mathcal{P}_{\varphi}:\mathcal{N}]\leq  [\mathcal{M}:\mathcal{N}].$$
By \cite[Corollary 1.5]{AGP26}, we have that 
\begin{align*}
    \#\mathfrak{S}_3\leq \#\{\text{intermediate subfactors of $\mathcal{N}\subseteq\mathcal{M}$}\}\leq 2^{\lfloor [\mathcal{M}:\mathcal{N}] \rfloor-1}.
\end{align*}
By Theorem \ref{thm:state equivalent}, 
\begin{align*}
    \#[\varphi]\leq \#\mathfrak{S}_2\times \#\mathfrak{S}_3\leq [\mathcal{M}:\mathcal{N}]2^{\lfloor [\mathcal{M}:\mathcal{N}] \rfloor-1}.
\end{align*}
This completes the proof.
\end{proof}
\begin{example}
    Suppose $G$ is a finite group outly acting on a factor $\mathcal{M}$ and $\mathcal{N}$ is a $G$-invariant subfactor of $\mathcal{M}$. 
    Then $\mathcal{N}\subseteq\mathcal{M}$ is an irreducible inclusion of factors with $[\mathcal{M}:\mathcal{N}]=|G|$.
    The intermediate subfactors are determined by some subgroup $H<G$.
    Let $P_H=\frac{1}{|H|}\sum_{h\in H}h$.
    Then $P_H$ is a biprojection.
    In this case, we have
\begin{align*}
    \#\{V \text{ is a bishift of biprojection}:\ V^*V=VV^*=P_H\}=|H|,
\end{align*}
and 
\begin{align*}
   \# \{Q\ \text{is biprojection}: Q\sim P\}=\#\{gHg^{-1}:g\in G\}\leq |G/H|.
\end{align*}
Hence $\#[\varphi]\leq |G|=[\mathcal{M}:\mathcal{N}]$.
\end{example}
Due to this group example, we conjecture that 
\begin{conjecture}\label{conj:conjecture}
    Suppose $\mathcal{N}\subseteq\mathcal{M}$ is an irreducible finite inclusion of factors and $P$ is a biprojection in $\mathcal{N}'\cap\mathcal{M}_1$.
    Then 
\begin{align*}
     \# \{Q\ \text{is biprojection}: Q\sim P\}\leq [\mathcal{M}:\mathcal{P}].
\end{align*}
\end{conjecture}
\begin{remark}\label{rem:conjecture}
If the conjecture is true, we have that $\#[\varphi]\leq[\mathcal{M}:\mathcal{N}]$.    
\end{remark}

\section*{Acknowlednements}
L. Huang was supported by National Natural Science Foundation of China (Grant No. 12501153).
C. Jiang was supported by Hebei Natural Science Foundation (No. A2023205045) and  by National Natural Science Foundation of China (Grant No. 12471120).
Z. Liu was supported by Beijing Natural Science Foundation (Grant No. Z221100002722017).
Z. Liu and J. Wu were supported by Beĳing Natural Science Foundation Key Program (Grant No. Z220002). 
J. Wu was supported by National Natural Science Foundation of China (Grant No. 12371124).

\bibliographystyle{plain}

\end{document}